\documentclass[11pt]{amsart}

\usepackage{fullpage,graphicx,amsfonts,amssymb,amsmath,amsthm}
\usepackage[all]{xy}
\usepackage[left=.8in,top=.8in,bottom=.8in,right=.8in,letterpaper]{geometry} 
\usepackage{mathtools}
\usepackage{graphicx}
\graphicspath{ {images/} }
\usepackage{enumerate}
\usepackage{hyperref} 
\usepackage{setspace}
\usepackage{amssymb} 
\usepackage{ esint }

\allowdisplaybreaks
\theoremstyle{plain} 
\newtheorem{theorem}    {Theorem}

\newtheorem{lemma}      [theorem]{Lemma}

\newtheorem{proposition}[theorem]{Proposition}

\theoremstyle{definition}

\theoremstyle{remark}

\usepackage{url}

\begin{document} 

\title{Stolarsky-Type Inequalities in a Max-Convolution Problem}

\author{Johannes Hosle \\ 06/06/2026}

\address{Department of Mathematics, Massachusetts Institute of Technology}
\email{jhosle@mit.edu}
\maketitle 

\begin{abstract}
For $m \in \mathbb{N}$, let $q_m := \frac{\log(2m+1)}{2\log(m+1)}$. The max-convolution inequality \begin{align*}
    \sum_{k=0}^{2m}\left(\max_{i+j=k} x_i y_j \right)^{q_m} &\ge \left(\sum_{i=0}^{m} x_i\right)^{q_m} \left(\sum_{j=0}^{m} y_j\right)^{q_m}
\end{align*}for arbitrary sequences $x_0 \ge x_1 \ge ... \ge x_m \ge 0, y_0 \ge y_1 \ge ... \ge y_m \ge 0$ implies an affirmative answer to a question of Bourgain, Dilworth, Ford, Konyagin, and Kutzarova \cite{BDFKK} on the sizes of sumsets in product sets. This inequality was proven for $m = 2$ by Becker, Ivanisvili, Krachun, and Madrid \cite{BIKM} by reducing the general case to the geometric block case via a max-tie analysis. We prove the geometric block case $x = (1, t, ..., t^{r}, 0, ..., 0)$ and $y = (1, t, ..., t^s, 0, ..., 0)$, $t \in [0, 1]$, for all $m \in \mathbb{N}$ via a comparison of Stolarsky means. Some perturbations are also verified. Finally, we prove the above inequality when one sequence has only two non-zero terms.
\end{abstract}

\section{Introduction}

Let $p_{m}$ denote the largest power such that the inequality \begin{align*}
    \left|A + B\right| &\ge (|A||B|)^{p_{m}}
\end{align*} holds for all subsets $A, B \subset \mathcal{C}_{m}^d := \{0, 1, 2, ..., m\}^d \subset \mathbb{Z}^d$, for all $d\ge 1$ where $m$ is a fixed positive integer. Choosing $A = B = \mathcal{C}_{m}^d$, we obtain \begin{align}\label{pmUpper}
    \frac{\log(2m + 1)}{2\log(m+1)} &\ge p_{m}.
\end{align}

For $m\ge 1,$ let $\tau_m$ be the solution of the equation \begin{align*}
    1 + (m(m+1))^{\tau_m} &= (m+1)^{2\tau_m}.
\end{align*} Bourgain, Dilworth, Ford, Konyagin, and Kutzarova \cite{BDFKK} proved that $p_{m} \ge \tau_{m}$, and expressed the opinion that equality in~\eqref{pmUpper} likely holds for all $m \in \mathbb{N}$. Note that both $\tau_{m}, \frac{\log(2m+1)}{2\log(m+1)} = \frac{1}{2} + \frac{\log 2}{2\log m} + o\left(\frac{1}{\log m}\right)$ as $m \to \infty,$ so~\eqref{pmUpper} is at least asymptotically optimal. Bourgain, Dilworth, Ford, Konyagin, and Kutzarova require this bound in the additive-combinatorial component of their explicit RIP matrix construction. 

They showed that in order to prove the inequality \begin{align*}
    \left|A + B\right| &\ge (|A||B|)^p
\end{align*} for all subsets $A, B \subset \mathcal{C}_{m}^d$, for all $d\ge 1$, for a fixed $p > 0$, it suffices to prove that for any sequences $x_0 \ge x_1 \ge ... \ge x_m \ge 0, y_0 \ge y_1 \ge ... \ge y_m \ge 0$, \begin{align}\label{maxconv}
    \sum_{k=0}^{2m}\left(\max_{i+j=k} x_i y_j \right)^{p} &\ge \left(\sum_{i=0}^{m} x_i\right)^p \left(\sum_{j=0}^{m} y_j\right)^p,
\end{align} The inequality~\eqref{maxconv} was proven for $m = 2, p = \frac{\log 5}{2 \log 3}$ by Becker, Ivanisvili, Krachun, and Madrid \cite{BIKM}, thereby confirming that equality holds in~\eqref{pmUpper} for $m = 2$. The argument in their work consists of a max-tie analysis, the final step being the confirmation of~\eqref{maxconv} for the geometric block case $x = (1, t, ..., t^{r}, 0, ..., 0)$ and $y = (1, t, ..., t^s, 0, ..., 0)$. Their $m = 2$ proof of this case uses Descartes's rule of signs, and does not extend already to $m = 3$.

The main result of this note confirms the geometric block case for all $m \in \mathbb{N}$.

\begin{theorem}\label{mainblockStolarsky}
For $m \in \mathbb{N}$, define $q_m := \frac{\log(2m+1)}{2\log(m+1)}$. For any integers $1\le r, s \le m$, we have the inequality \begin{align*}
    \sum_{k=0}^{r+s}t^{kq_m} &\ge \left(\sum_{i = 0}^{r} t^i\right)^{q_m} \left(\sum_{j=0}^{s} t^j\right)^{q_m}
\end{align*} for all $t \in [0, 1].$ 
\end{theorem}

Our proof makes use of a comparison result for Stolarsky means.

The next two results correspond to the case $x = y = (1, t, ..., t^r, t^r u)$ and $x = y = (1, t, tu, t^2 u, ..., t^{r-1} u)$ respectively, where we use here and henceforth the convention that if the last entries of a vector all vanish, we omit them. They can be viewed as perturbed geometric block cases, and are also proven and used for $m = 2$ by Becker, Ivanisvili, Krachun, and Madrid \cite{BIKM}.

\begin{proposition}\label{pert1}
If $1\le r\le m-1$, then \begin{align*}
    \sum_{k=0}^{2r}t^{kq_m} + t^{2rq_m}(u^{q_m} + u^{2q_m}) &\ge \left(\sum_{i=0}^{r} t^i + t^r u \right)^{2q_m},
\end{align*} for all $0\le u \le t \le 1.$
\end{proposition}

\begin{proposition}\label{pert2}
If $2\le r\le m$, then \begin{align*}
    1 + t^{q_m} + t^{2q_m} + u^{q_m} \sum_{k=2}^{r} t^{kq_m} + u^{2q_m} \sum_{k=r}^{2r-2} t^{kq_m} &\ge \left(1 + t + u \sum_{i=1}^{r-1} t^i \right)^{2q_m}.
\end{align*}
\end{proposition}

\section{Stolarsky means comparison inequality}

We will require the notion of a Stolarsky mean $S_{p, q}(a, b)$ for $p, q \in \mathbb{R}$ and $a, b \in \mathbb{R}_{+}$. These were introduced by Stolarsky \cite{Stolarsky}. When $pq(p-q) \neq 0$ and $a\neq b$, \begin{align*}
    S_{p, q}(a, b) :=  \left(\frac{q(a^p - b^p)}{p(a^q - b^q)} \right)^{\frac{1}{p-q}}.
\end{align*}When $a = b$, $S_{p, q}(a, a) = a.$ The following comparison result for Stolarsky means was first proven by Leach and Sholander \cite{LeachSholanderII} and later by P\'ales \cite{Pales}.

\begin{theorem}\label{StolarskyComparison} (\cite{LeachSholanderII}, \cite{Pales}) The comparison inequality \begin{align*}
    S_{p, q}(a, b) &\le S_{r, s}(a, b)
\end{align*} holds for all $a, b \in \mathbb{R}_{+}$ if and only if all the following conditions are met: \begin{enumerate}
    \item $p + q \le r + s$
    \item $l(p, q) \le l(r, s)$ if $\min(p, q, r, s)\ge 0$ or $\max(p, q, r, s) \le 0$, where $l(u, v) = \begin{cases}
        \frac{u - v}{\log \frac{u}{v} } & \text{ if }uv > 0 \\
        0 &\text{ if }uv = 0.
    \end{cases}$
    \item $\mu(p, q) \le \mu(r, s)$ if $\min(p, q, r, s) < 0 < \max(p, q, r, s)$, where $\mu(u, v) = \begin{cases}
        \frac{|u| - |v|}{u - v} &\text{ if }u\neq v \\ 
        \text{sgn}(u) &\text{ if }u = v.
    \end{cases}$
    \end{enumerate}
\end{theorem}We remark now that when $b > a > 0$ and $t = \log \sqrt{\frac{b}{a}} \in (0, \infty)$, the Stolarsky mean can be expressed by a hyperbolic function \begin{align*}
    S_{p, q}(a, b) = \sqrt{ab} H_{p, q}(t),
\end{align*}where $H_{p, q}(t) = \left(\frac{q \sinh(pt)}{p \sinh(qt)} \right)^{\frac{1}{p-q}}.$

\section{Proof of Theorem~\ref{mainblockStolarsky}}

\begin{lemma}\label{qmdecreasing}
The exponent $q_m$ is decreasing for $m\ge 1$. 
\begin{proof} Indeed, treating $m$ as a real variable, \begin{align*}
    q_m' = \frac{\frac{2}{m+1}\log(m+1) - \frac{1}{m+1}\log(2m+1)}{2\log^2(m+1)}.
\end{align*} Thus, $q_m' < 0$ if and only if \begin{align*}
    \frac{2(m+1)}{2m+1} \log(m+1) < \log(2m+1).
\end{align*} Define $F(m) := \log(2m+1) - \frac{2(m+1)}{2m+1} \log(m+1).$ Then, \begin{equation*}
    F'(m) = \frac{2\log(m+1)}{(2m+1)^2} > 0.
\end{equation*} Since $F(1) = \log 3 - \frac{4}{3} \log 2 > 0$, $F(m) \ge 0$ for all $m\ge 1,$ completing the proof that $q_m$ is decreasing.
\end{proof}
\end{lemma}

\begin{proof}[Proof of Theorem~\ref{mainblockStolarsky}]
To begin, we consider the case $r = s = m$, so that we must prove \begin{align*}
    \sum_{k=0}^{2m} t^{kq_m} &\ge \left( \sum_{k=0}^{m} t^k\right)^{2q_m}.
\end{align*}Put $t = e^{-2u}$ for $u \ge 0$. The desired inequality is then \begin{align*}
    \sum_{k=0}^{2m}e^{-2q_m u k} &\ge \left(\sum_{k=0}^{m} e^{-2uk} \right)^{2q_m}
\end{align*} Note that by the geometric series formula, we have \begin{align*}
    \sum_{k=0}^{n-1}e^{-ku} &= e^{-\frac{(n-1)u}{2}} \frac{\sinh\left(\frac{nu}{2}\right)}{\sinh\left(\frac{u}{2}\right)}.
\end{align*} Therefore, our desired inequality is \begin{align*}
    \frac{\sinh((2m+1)q_m u)}{\sinh(q_m u)} &\ge \left(\frac{\sinh((m+1)u)}{\sinh(u)} \right)^{2q_m},
\end{align*} or equivalently, \begin{align*}
    \frac{\sinh((2m+1)q_m u)}{(2m+1)\sinh(q_m u)} &\ge \left(\frac{\sinh((m+1)u)}{(m+1)\sinh(u)} \right)^{2q_m}.
\end{align*} Rewriting in the notation of Stolarsky means, the inequality is equivalent to \begin{align*}
    H_{m+1, 1}(u)^{q_m} &\le H_{2m+1, 1}(q_m u).
\end{align*} By the elementary scaling identity \begin{align*}
    H_{\lambda a, \lambda b}(x) = H_{a, b}(\lambda x)^{\frac{1}{\lambda}},
\end{align*} it remains to show that \begin{align*}
    H_{\frac{m+1}{q_m}, \frac{1}{q_m}}(q_m u) &\le H_{2m+1, 1}(q_m u).
\end{align*}

To finish the proof, we now verify the conditions of the Stolarsky mean comparison theorem, Theorem~\ref{StolarskyComparison}, for $(p, q) = \left(\frac{m+1}{q_m}, \frac{1}{q_m}\right)$ and $(r, s) = (2m+1, 1)$. Since all four parameters are positive, the third condition does not apply. It therefore suffices to prove \begin{align}\label{sumcondition}
    \frac{m+1}{q_m} + \frac{1}{q_m} &\le 2m+2
\end{align} and \begin{align}\label{lcondition}
    l\left(\frac{m+1}{q_m}, \frac{1}{q_m}\right) &\le l\left(2m+1, 1\right).
\end{align} We begin with the first inequality. We must show that $q_m \ge \frac{m+2}{2m+2}$. Let $n = m+1$. Then, the inequality is equivalent to \begin{align*}
    n\log(2n-1) \ge (n+1)\log n.
\end{align*} Write $n\log(2n-1) = n\log n + n\log\left(2-\frac{1}{n}\right).$ It suffices to therefore prove $n\log\left(2 - \frac{1}{n}\right) \ge \log n,$ or $\left(2 - \frac{1}{n}\right)^n \ge n$. Since, for $n \ge 2$, $2 - \frac{1}{2} \ge \frac{3}{2}$ and $\left(\frac{3}{2}\right)^n \ge n$, the validity of~\eqref{sumcondition} follows.

We turn now to~\eqref{lcondition}. We have \begin{equation*}
    l\left(\frac{m+1}{q_m}, \frac{1}{q_m}\right) = \frac{\frac{m}{q_m}}{\log(m+1)} = \frac{2m}{\log(2m+1)} = l(2m+1, 1),
\end{equation*} so~\eqref{lcondition} is in fact an equality, completing the proof of this $r = s = m$ case.

We move on to the general case $1\le r, s, \le m$. Let \begin{align*}
    G_N(t) := \sum_{k=0}^{N} t^k = \frac{1 - t^{N+1}}{1 - t}.
\end{align*} We seek to prove \begin{align}\label{Gineq}
    G_{r+s}(t^{q_m}) &\ge G_{r}(t)^{q_m} G_{s}(t)^{q_m}.
\end{align} In the first part of this proof, we showed that \begin{align}\label{Gineqrsm}
    G_{2N}(t^{q_N}) &\ge G_{N}(t)^{2q_N}
\end{align} for all $N \ge 1$. 

By Lemma~\ref{qmdecreasing}, $q_m \le q_N$ for $N \le m$. Set $\alpha = \frac{q_m}{q_N} \le 1.$ Since $0 \le t\le 1, t^{\alpha} \ge t$. By~\eqref{Gineqrsm} for $t^{\alpha}$, we obtain \begin{equation*}
    G_{2N}(t^{q_m}) = G_{2N}((t^{\alpha})^{q_N}) \ge G_{N}(t^{\alpha})^{2q_N}.
\end{equation*} Since $G_N(t^{\alpha}) \ge G_N(t) \ge 1$ and recalling that $q_m \le q_N$, we obtain \begin{align}\label{GNqm}
    G_{2N}(t^{q_m}) &\ge G_{N}(t)^{2q_m}.
\end{align} Observe now that $N \to G_N(t)$ is log-concave for any fixed $t \in (0, 1)$. Indeed, we may compute \begin{equation*}
    \frac{\partial^2}{\partial N^2} \log \frac{1-t^{N+1}}{1-t} = -\frac{t^{N+1} (\log t)^2}{(1-t^{N+1})^2} < 0.
\end{equation*} Therefore, \begin{align*}
    \log G_{r+s}(t^{q_m}) &\ge \frac{1}{2}\log G_{2r}(t^{q_m}) + \frac{1}{2}\log G_{2s}(t^{q_m}),
\end{align*} or equivalently, \begin{align*}
    G_{r+s}(t^{q_m}) \ge \sqrt{G_{2r}(t^{q_m})G_{2s}(t^{q_m})}.
\end{align*} By~\eqref{GNqm}, \begin{align*}
    G_{r+s}(t^{q_m}) \ge G_{r}(t)^{q_m}G_s(t)^{q_m},
\end{align*} which is exactly the statement~\eqref{Gineq} we set out to prove.

\end{proof}

\section{Proof of perturbed block inequalities}

We begin with the following lemma.

\begin{lemma}\label{lemmaineq1}
For $p \in [1, 2]$ and $s \in [0, 1]$, we have \begin{align*}
    1 - s^p &\ge (1-s)(1+s)^{p-1}.
\end{align*}
\begin{proof}
The endpoint cases \(s=0\) and \(s=1\) are immediate. For \(s\in(0,1)\), set
\[
    t:=\frac{1-s}{1+s}\in(0,1).
\]
Equivalently,
\[
    s=\frac{1-t}{1+t}.
\]
The desired inequality is equivalent, after multiplying by \((1+t)^p\), to
\[
    (1+t)^p-(1-t)^p \ge 2^p t.
\]
Define
\[
    H(t):=(1+t)^p-(1-t)^p-2^p t,
    \qquad 0\le t\le 1.
\]
Then
\[
    H(0)=H(1)=0.
\]
Moreover,
\[
    H''(t)
    =
    p(p-1)\left((1+t)^{p-2}-(1-t)^{p-2}\right).
\]
The factor $p(p-1) \ge 0$, since $p\ge 1$. Additionally, since \(p\le 2\), we have \(p-2\le 0\). Therefore, for \(0<t<1\),
\[
    (1+t)^{p-2}\le (1-t)^{p-2}.
\]
Hence
\[
    H''(t)\le 0,
\]
and \(H\) is concave on \([0,1]\). Since a concave function lies above the chord joining its endpoint values, and since \(H(0)=H(1)=0\), we get
\[
    H(t)\ge 0
\]
for all \(t\in[0,1]\). Therefore,
\[
    (1+t)^p-(1-t)^p \ge 2^p t,
\] which is equivalent to the desired inequality.

\end{proof}
\end{lemma}

\begin{proof}[Proof of Proposition~\ref{pert1}]
Set \begin{equation*}
    A := \sum_{k=0}^{2r} t^{kq_m}, B := \sum_{i=0}^{r} t^i, 
\end{equation*} Define $c := t^r u, d:= t^r$. The desired inequality is \begin{align*}
    A + (dc)^{q_m} + c^{2q_m} &\ge (B + c)^{2q_m},
\end{align*} for $0 \le c \le t^{r+1} = dt.$ Define \begin{align*}
    H(c) = A + (dc)^{q_m} + c^{2q_m} - (B + c)^{2q_m}.
\end{align*} Note that the inequalities $H(0) \ge 0, H(dt) \ge 0$ follow from Theorem~\ref{mainblockStolarsky} for $r, s = m-1$ and $r, s = m$ respectively. It remains to rule out a negative interior minimum. Let $c \in (0, dt)$ be a critical point of $H$, that is, $H'(c) = 0.$ We have \begin{equation*}
    H'(c) = q_m  d^{q_m} c^{q_m - 1} + 2q_m c^{2q_m - 1} - 2q_m(B + c)^{2q_m-1}.
\end{equation*} Setting this to zero, we obtain \begin{align*}
    (dc)^{q_m} + 2c^{2q_m} &=  2c(B+c)^{2q_m-1}.
\end{align*} We now write \begin{align*}
    H(c) - H(0) &= (dc)^{q_m} + c^{2q_m} - (B+c)^{2q_m} + B^{2q_m} \\ &= 2c(B+c)^{2q_m-1} - c^{2q_m} - (B+c)^{2q_m} + B^{2q_m}
\end{align*} Set $s := \frac{c}{B}$. Since $c \le 1$ and $B \ge 1$, $s \in [0, 1].$ Then, \begin{align*}
    \frac{1}{B^{2q_m}}(H(c) - H(0)) &= 2s(1+s)^{2q_m - 1} - s^{2q_m} - (1+s)^{2q_m} + 1 \\ &= (s-1)(1+s)^{2q_m - 1} - s^{2q_m} + 1.
\end{align*} In order to conclude that this is nonnegative by Lemma~\ref{lemmaineq1}, it suffices to verify that $q_m \in \left[\frac{1}{2}, 1\right].$ By Lemma~\ref{qmdecreasing}, $q_m$ is decreasing for $m \ge 1$. Therefore, \begin{align*}
    1 \ge q_1 \ge q_m \ge \lim_{m\to \infty}q_m = \frac{1}{2},
\end{align*} as desired. It follows from Lemma~\ref{lemmaineq1} that $H(c) \ge H(0) \ge 0$, completing the proof.
\end{proof}

\begin{proof}[Proof of Proposition~\ref{pert2}] The case $r = 2$ is contained in Proposition~\ref{pert1}. We may therefore consider $r\ge 3.$ Set \begin{equation*}
    A := \sum_{k=2}^{r} t^{kq_m}, B := \sum_{k=r}^{2r-2} t^{kq_m}, C := \sum_{i=1}^{r-1}t^i.
\end{equation*} For fixed $t > 0$, define \begin{align*}
    H(u) &:= 1 + t^{q_m} + t^{2q_m} + A u^{q_m} + B u^{2q_m}  - \left(1 + t + Cu\right)^{2q_m}
\end{align*} for $u \in [0, t]$. Note that both the inequalities $H(0), H(t) \ge 0$ follow from Theorem~\ref{mainblockStolarsky}. It remains to rule out a negative interior minimum. A critical point $u$ of $H$ satisfies \begin{equation*}
    H'(u) = q_m A u^{q_m - 1} + 2q_m B u^{2q_m - 1} - 2q_m C(1+t+Cu)^{2q_m - 1} = 0,
\end{equation*} or \begin{align*}
    A + 2Bu^{q_m} &= 2Cu^{1-q_m}(1+t+Cu)^{2q_m - 1}.
\end{align*} Set $F(u) = A + 2Bu^{q_m}, G(u) = 2Cu^{1-q_m}(1+t+Cu)^{2q_m - 1}.$ Since $q_m < 1$, \begin{align*}
    \frac{G(0)}{F(0)} = 0.
\end{align*}Therefore, $H'(u) > 0$ near $0$. We now show that $\frac{G}{F}$ is increasing on $[0, t]$. Note that this suffices for the proof, because then $\frac{G}{F} = 1$ at at most one place, and $H'$ changes sign from positive to negative. In particular, a critical point of $H$ will not correspond to a minimum, let alone a negative interior minimum. To this end, we compute \begin{align*}
    &u \frac{d}{du} \log G(u) = 1 - q_m + (2q_m - 1) \frac{Cu}{1+t+Cu}, \\ &u \frac{d}{du}\log F(u) = \frac{2q_m Bu^{q_m}}{A + 2Bu^{q_m}}.
\end{align*}Set $Q := \frac{Cu}{1+t}$, so that \begin{align*}
    u\frac{d}{du}\log G(u) &= 1-q_m + (2q_m - 1)\frac{Q}{1+Q}.
\end{align*}

Since $B = t^{(r-2)q_m}A$, \begin{align*}
    \frac{Bu^{q_m}}{A} &= t^{(r-2)q_m} u^{q_m}.
\end{align*} Since we are assuming $r \ge 3$, \begin{equation*}
    (1+t)t^{r-2} = t^{r-2} + t^{r-1} \le t + t^2 + ... + t^{r-1} = C.
\end{equation*}Equivalently, \begin{align*}
    t^{r-2}u &\le Q.
\end{align*} Since the function $x \to \frac{2x}{1+2x}$ is increasing on $[0, \infty)$, \begin{equation*}
    u\frac{d}{du}\log F(u) \le \frac{2q_m Q^{q_m}}{1 + 2Q^{q_m}}.
\end{equation*} It remains to prove \begin{align}\label{pert2finalineq}
    1 - q_m + (2q_m - 1) \frac{Q}{1+Q} - \frac{2q_m Q^{q_m}}{1+2Q^{q_m}} &> 0.
\end{align}Clearing denominators, the inequality is equivalent to \begin{align*}
    (1-q_m) + q_m Q &\ge 2(2q_m - 1)Q^{q_m}.
\end{align*} We have \begin{align*}
    \frac{d}{dQ}\left(\frac{(1-q_m) + q_m Q}{Q^{q_m}}\right) &= \frac{q_m(1-q_m)(Q-1)}{Q^{q_m + 1}},
\end{align*} so that $\frac{(1-q_m) + q_m Q}{Q^{q_m}}$ achieves its minimum at $Q = 1$, from which we obtain \begin{align*}
    \frac{(1-q_m) + q_m Q}{Q^{q_m}} &\ge 1
\end{align*} for all $Q > 0.$ To conclude~\eqref{pert2finalineq} and the proposition, it suffices to observe, by Lemma~\ref{qmdecreasing}, that \begin{equation*}
    q_m \le q_3 < \frac{3}{4},
\end{equation*} so that \begin{align*}
    2(2q_m - 1) < 1.
\end{align*}
\end{proof}

\section{Max-tie reduction}

Note that the inequality~\eqref{maxconv} is homogeneous in $x$ and $y$, so we can assume that both $\sum_{i=0}^{m}x_i = \sum_{j=0}^{m}y_j = 1$. We denote \begin{align*}
    \mathbf{P}_m := \left\{x = (x_0, ..., x_m) \in [0, 1]^{m+1}: \sum_{i=0}^{m} x_i = 1\right\}.
\end{align*} The following result was proven for $m = 2$ by Becker, Ivanisvili, Krachun, and Madrid \cite{BIKM}. Their proof extends directly here. We provide it for completeness.

\begin{lemma}\label{maxtielemma} Suppose that $(x, y) \in \mathbf{P}_m^2$ minimizes \begin{align*}
    J_p(x, y) := \sum_{k=0}^{2m}\max_{i+j = k} x_i^p y_j^p
\end{align*} on $\mathbf{P}_m^2$ for a fixed $p \in (0, 1).$ Then, for each $i$, we have $x_i = 0$ or $x_i = 1$ or there exists $k$ and $i' \neq i$ such that $x_i y_{k - i} = x_{i'}y_{k-i'}.$ Similarly, for each $j$, we have $y_j = 0$ or $y_j = 1$ or there exists $k$ and $j' \neq j$ such that $x_{k-j} y_j = x_{k-j'} y_{j'}.$
\begin{proof}
Suppose that $(x, y)$ is a minimizer without the claimed property. By symmetry in $x$ and $y$, we can assume that there exists $i$ such that $x_i \in (0, 1)$ and $x_i y_{k-i} \neq x_{i'}y_{k-i'}$ for all $k = i, ..., i+m$ and $i' \neq i$. We now perturb $x$ to decrease $J_p(x, y)$ as follows. We define \begin{align*}
    x_n(t) := \begin{cases}
        (1+t)x_n &\text{   if }n\neq i \\
        x_i - t(1-x_i) &\text{   if }n = i.
    \end{cases}
\end{align*} Since $x_i \in (0, 1)$, there exists $\varepsilon > 0$ sufficiently small so that $x(t) \in \mathbf{P}_m$ for all $t \in (-\varepsilon, \varepsilon)$. For each $k$ with $x_i y_{k-i} < \max_{i' + j' = k}x_{i'}y_{j'}$, there exists another $\varepsilon > 0$ such that \begin{align*}
    \max_{i' + j' = k}x_{i'}(t)^p y_{j'}^p &= (1+t)^p \max_{i' + j' = k} x_{i'}^p y_{j'}^p.
\end{align*} Indeed, each product $x_{i'} y_{k - i'}$ with $i' \neq i$ changes by the same factor and the product $x_i y_{k-i}$ is strictly smaller. Finally, if $x_i y_{k-i} = \max_{i'+j' = k} x_{i'}y_{j'}$, then there also exists $\varepsilon > 0$ so that \begin{align*}
    x_i(t)^p y_{k-i}^p &= \max_{i' + j' = k} x_{i'}(t)^p y_{j'}^p,
\end{align*} since $x_i y_{k-i}$ is strictly larger than all other products $x_{i'} y_{k-i'}$. Taking the minimum of all $\varepsilon$'s, we obtain a neighborhood $(-\varepsilon, \varepsilon)$ on which $J_p(x(t), y)$ is concave as a function of $t$. Thus, $(x, y) = (x(0), y)$ cannot be a minimizer of $J_p$.
\end{proof}
\end{lemma}

\section{Case of $x_2 = 0$}

\begin{theorem}Let $m \in \mathbb{N}$. For every
\[
    x_0\ge x_1 \ge 0,
    \qquad
    y_0\ge y_1\ge\cdots\ge y_m \ge 0,
\]
we have
\[
    \sum_{k=0}^{m+1}
    \left(
        \max_{\substack{i+j=k\\0\le i\le1\\0\le j\le m}}
        x_i y_j
    \right)^{q_m}
    \ge
    (x_0+x_1)^{q_m}
    \left(\sum_{j=0}^{m}y_j\right)^{q_m} .
\]
\begin{proof} We may assume that $x_0 + x_1 = \sum_{j=0}^{m} y_j = 1.$ Arguing as in the proof of Lemma~\ref{maxtielemma}, such a minimizer $(x, y)$ of $\sum_{k=0}^{m+1}
\left(
    \max_{\substack{i+j=k\\0\le i\le1\\0\le j\le m}}
    x_i y_j
\right)^{q_m}$ will satisfy the condition that every $x_i, y_j$ either is in $\{0, 1\}$ or participates in a max-tie. If $x_1 = 0$, then the conclusion follows by subadditivity of $u \to u^{q_m}$. Therefore, we may assume that $x_0 = \frac{1}{1+t}, x_1 = \frac{t}{1+t},$ where $0 < t\le 1.$ Let $N := \max\{j: y_j > 0\},$ so that 
\[
    y_0,\dots,y_N>0,
    \qquad
    y_{N+1}=\cdots=y_m=0.
\]
For each \(1\le j\le N\), at least one adjacent
constraint must be active. Consequently,
\[
    \frac{y_j}{y_{j-1}}\in\{1,t\},
    \qquad 1\le j\le N.
\]

For such a sequence,
\[
    \max\{y_j,ty_{j-1}\}=y_j,
    \qquad 1\le j\le N,
\]
and therefore the desired inequality becomes
\begin{align}\label{zineqR1m}
    \sum_{j=0}^{N}y_j^{q_m}+(t y_N)^{q_m}
    &\ge
    (1+t)^{q_m}\left(\sum_{j=0}^{N}y_j\right)^{q_m}.
\end{align} Suppose first that
\[
    y_j=t^j y_0,\qquad 0\le j\le N.
\]
Then~\eqref{zineqR1m} becomes
\[
    y_0^{q_m}\sum_{k=0}^{N+1}t^{kq_m}
    \ge
    y_0^{q_m}(1+t)^{q_m}
    \left(\sum_{j=0}^{N}t^j\right)^{q_m},
\]
which is contained in Theorem~\ref{mainblockStolarsky}.

It remains to show that vector $y$ can be transformed to this geometric form without increasing
the defect. Define
\[
    \mathcal D(y_0,\dots,y_N)
    :=
    \sum_{j=0}^{N}y_j^{q_m}+(t y_N)^{q_m}
    -
    (1+t)^{q_m}\left(\sum_{j=0}^{N}y_j\right)^{q_m}.
\]
Assume that \(y_{j+1}=y_j\) for some \(0\le j<N\). Write
\[
    A=\sum_{i=0}^{j}y_i,
    \qquad
    B=\sum_{i=j+1}^{N}y_i,
\]
and
\[
    A_{q_m}=\sum_{i=0}^{j}y_i^{q_m},
    \qquad
    B_{q_m}=\sum_{i=j+1}^{N}y_i^{q_m}+(t y_N)^{q_m}.
\]
For \(\theta\in[t,1]\), define
\[
    y^{(\theta)}
    =
    (y_0,\dots,y_j,\theta y_{j+1},\dots,\theta y_N).
\]
Then
\[
    \mathcal D(y^{(\theta)})
    =
    A_{q_m}+\theta^{q_m} B_{q_m}
    -
    (1+t)^{q_m}(A+\theta B)^{q_m}.
\]

We now prove the inequality \begin{align}\label{Bqmineq}
    B_{q_m} &\ge (1+t)^{q_m} B^{q_m}.
\end{align}Let 
\[
    z_{\ell} := y_{j+1+\ell},\qquad 0\le \ell\le L,
\]where \begin{align*}
    L &= N - j - 1.
\end{align*}Then,
\[
    B = \sum_{\ell = 0}^{L}z_{\ell},
    \qquad
    B_{q_m}=\sum_{\ell = 0}^{L}z_{\ell}^{q_m} + (tz_L)^{q_m}.
\] The tail still satisfies
\[
    z_0 \ge z_1 \ge ... \ge z_L > 0,
    \qquad\frac{z_{\ell}}{z_{\ell-1}} \in \{1, t\}, \text{   }1\le \ell \le L.
\] Recall \begin{align*}
    G_L := \sum_{\ell = 0}^{L} t^{\ell}.
\end{align*} We now compare $z$ with the geometric tail with the same mass,
\[
    g_{\ell} := \frac{Bt^{\ell}}{G_L}
    \qquad
    0 \le \ell \le L.
\] We prove that $g$ majorizes $z$. Indeed, let \begin{align*}
    P_k := \sum_{\ell = 0}^{k}z_{\ell}.
\end{align*} Since $z_{\ell+1} \ge tz_{\ell}$, for any $0 \le \ell \le k \le L,$ \begin{align*}
    z_{\ell} \le t^{-(k-\ell)} z_k.
\end{align*} Therefore, \begin{equation*}
    P_k \le z_k \sum_{\ell = 0}^{k}t^{-(k-\ell)} = z_k t^{-k} G_k.
\end{equation*} On the other hand, \begin{equation*}
    B - P_k = \sum_{\ell = k+1}^{L} z_{\ell} \ge z_k \sum_{r = 1}^{L-k}t^r = z_k(G_{L-k} - 1).
\end{equation*} Combining the two inequalities yields \begin{align*}
    B - P_k &\ge P_k \frac{G_{L-k} - 1}{t^{-k} G_k}.
\end{align*} Hence, \begin{equation*}
    B \ge P_k \frac{t^{-k}G_k + G_{L-k} - 1}{t^{-k} G_k} = P_k \frac{t^{-k} G_{L}}{t^{-k}G_k},
\end{equation*} or equivalently, \begin{equation}\label{majorization}
    P_k \le \frac{G_k}{G_L} B = \sum_{\ell = 0}^{k}g_{\ell}.
\end{equation} Thus, $g$ majorizes $z$ as desired. Since $q_m < 1$, $u \to u^{q_m}$ is concave. By Karamata's inequality, \begin{equation*}
    \sum_{\ell = 0}^{L}z_{\ell}^{q_m} \ge \sum_{\ell = 0}^{L}g_{\ell}^{q_m} = \frac{B^{q_m}}{G_L^{q_m}} \sum_{\ell = 0}^{L} t^{\ell q_m}.
\end{equation*} Additionally, by~\eqref{majorization} for $k = L - 1$, we obtain \begin{equation*}
    z_L = B - P_{L-1} \ge B - \frac{G_{L-1}}{G_L}B = B \frac{t^L}{G_L}.
\end{equation*} Therefore, \begin{equation*}
    B_{q_m} = \sum_{\ell = 0}^{L} z_{\ell}^{q_m} + (tz_L)^{q_m} \ge \frac{B^{q_m}}{G_L^{q_m}}\sum_{\ell = 0}^{L+1}t^{\ell q_m}.
\end{equation*} By Theorem~\ref{mainblockStolarsky}, \begin{align*}
    \sum_{\ell = 0}^{L+1}t^{\ell q_m} &\ge (1+t)^{q_m} G_L^{q_m},
\end{align*} so that \begin{align*}
    B_{q_m} &\ge (1+t)^{q_m} B^{q_m},
\end{align*} which is~\eqref{Bqmineq}. Therefore,
\begin{align*}
    \frac{d}{d\theta}\mathcal D(y^{(\theta)})
    &=
    q_m\theta^{q_m-1}B_{q_m}
    -
    q_m(1+t)^{q_m} B(A+\theta B)^{q_m-1} \\
    &\ge
    q_m(1+t)^{q_m}
    \left[
        \theta^{q_m-1}B^{q_m}
        -
        B(A+\theta B)^{q_m-1}
    \right] \\
    &=
    q_m(1+t)^{q_m} B
    \left[
        (\theta B)^{q_m-1}
        -
        (A+\theta B)^{q_m-1}
    \right] \\
    &\ge 0,
\end{align*}
where in the last step we used $q_m < 1.$
Thus
\[
    \mathcal D(y^{(1)})\ge \mathcal D(y^{(t)}).
\]
Replacing \(y\) by \(y^{(t)}\) changes the plateau ratio
\[
    \frac{y_{j+1}}{y_j}=1
\]
into the geometric ratio
\[
    \frac{t y_{j+1}}{y_j}=t,
\]
while preserving the property that all adjacent ratios belong to
\(\{1,t\}\). Hence, each such compression removes one plateau and does not
increase the defect.

Iterating, we obtain
\[
    \mathcal D(y_0,\dots,y_N)
    \ge
    \mathcal D(y_0,ty_0,t^2y_0,\dots,t^Ny_0).
\]
The latter quantity is nonnegative by Theorem~\ref{mainblockStolarsky}, as discussed above. This completes the proof.
\end{proof}
\end{theorem}

\section{Acknowledgements}

I thank David Jerison for helpful conversations on this paper. The author was supported in part by Simons Foundation Collaboration Grant 601948 DJ. The author used ChatGPT Plus 5.4 and 5.5 during the development of this work, including for suggesting and ruling out proof strategies and assisting with calculations.

\bibliographystyle{alpha}
\bibliography{biblio.bib}

@article {BIKM,
    AUTHOR = {Becker, Lars and Ivanisvili, Paata and Krachun, Dmitry and
              Madrid, Jos\'e},
     TITLE = {Discrete {B}runn-{M}inkowski inequality for subsets of the
              cube},
   JOURNAL = {Combinatorica},
  FJOURNAL = {Combinatorica. An International Journal on Combinatorics and
              the Theory of Computing},
    VOLUME = {45},
      YEAR = {2025},
    NUMBER = {5},
     PAGES = {Paper No. 48, 25},
      ISSN = {0209-9683,1439-6912},
   MRCLASS = {11B30 (11B13)},
  MRNUMBER = {4964355},
       DOI = {10.1007/s00493-025-00180-0},
       URL = {https://doi.org/10.1007/s00493-025-00180-0},
}

@article {LeachSholanderII,
    AUTHOR = {Leach, E. B. and Sholander, M. C.},
     TITLE = {Extended mean values. {II}},
   JOURNAL = {J. Math. Anal. Appl.},
  FJOURNAL = {Journal of Mathematical Analysis and Applications},
    VOLUME = {92},
      YEAR = {1983},
    NUMBER = {1},
     PAGES = {207--223},
      ISSN = {0022-247X},
   MRCLASS = {26A24},
  MRNUMBER = {694172},
MRREVIEWER = {A.\ Lupa\c s},
       DOI = {10.1016/0022-247X(83)90280-9},
       URL = {https://doi.org/10.1016/0022-247X(83)90280-9},
}

@article {Pales,
    AUTHOR = {P\'ales, Zsolt},
     TITLE = {Inequalities for differences of powers},
   JOURNAL = {J. Math. Anal. Appl.},
  FJOURNAL = {Journal of Mathematical Analysis and Applications},
    VOLUME = {131},
      YEAR = {1988},
    NUMBER = {1},
     PAGES = {271--281},
      ISSN = {0022-247X},
   MRCLASS = {26D15},
  MRNUMBER = {934446},
MRREVIEWER = {Wo-Sang\ Young},
       DOI = {10.1016/0022-247X(88)90205-3},
       URL = {https://doi.org/10.1016/0022-247X(88)90205-3},
}

@article {BDFKK,
    AUTHOR = {Bourgain, Jean and Dilworth, Stephen and Ford, Kevin and
              Konyagin, Sergei and Kutzarova, Denka},
     TITLE = {Explicit constructions of {RIP} matrices and related problems},
   JOURNAL = {Duke Math. J.},
  FJOURNAL = {Duke Mathematical Journal},
    VOLUME = {159},
      YEAR = {2011},
    NUMBER = {1},
     PAGES = {145--185},
      ISSN = {0012-7094,1547-7398},
   MRCLASS = {11B30 (11B13 11T23 41A46 94A12)},
  MRNUMBER = {2817651},
MRREVIEWER = {Moubariz\ Z.\ Garaev},
       DOI = {10.1215/00127094-1384809},
       URL = {https://doi.org/10.1215/00127094-1384809},
}

@article {Stolarsky,
    AUTHOR = {Stolarsky, Kenneth B.},
     TITLE = {Generalizations of the logarithmic mean},
   JOURNAL = {Math. Mag.},
  FJOURNAL = {Mathematics Magazine},
    VOLUME = {48},
      YEAR = {1975},
     PAGES = {87--92},
      ISSN = {0025-570X,1930-0980},
   MRCLASS = {26A87},
  MRNUMBER = {357718},
MRREVIEWER = {V.\ Ganapathy Iyer},
       DOI = {10.2307/2689825},
       URL = {https://doi.org/10.2307/2689825},
}
\end{document}